\documentclass[final,12pt]{article}%
\usepackage{amsmath}
\usepackage{amssymb,amscd}
\usepackage{xypic}
\usepackage{amsthm}
\usepackage{amsfonts}
\usepackage{graphicx}
\usepackage{amssymb}
\usepackage{url}%
\newtheorem{X}{X}[section]
\theoremstyle{plain}
\newtheorem{corollary}[X]{Corollary}

\newtheorem{proposition}[X]{Proposition}
\newtheorem{theorem}[X]{Theorem}
\theoremstyle{definition}
\newtheorem{plain}[X]{}

\newtheorem{example}[X]{Example}

\newtheorem{remark}[X]{Remark}

\DeclareMathOperator{\inv}{inv}
\DeclareMathOperator{\SL}{SL}
\DeclareMathOperator{\Lie}{Lie}
\DeclareMathOperator{\ad}{ad}
\DeclareMathOperator{\der}{der}

\DeclareMathOperator{\End}{End}
\DeclareMathOperator{\Aut}{Aut}
\DeclareMathOperator{\GL}{GL}
\DeclareMathOperator{\Hom}{Hom}

\DeclareMathOperator{\Gal}{Gal}

\begin{document}

\title{Nonhomeomorphic conjugates of connected Shimura varieties}
\author{James S. Milne and Junecue Suh}
\date{April 13, 2008}
\maketitle

\begin{abstract}
We show that conjugation by an automorphism of $\mathbb{C}$ may change the
topological fundamental group of a locally symmetric variety over $\mathbb{C}%
$. As a consequence, we obtain a large class of algebraic varieties defined
over number fields with the property that different embeddings of the number
field into $\mathbb{C}$ give complex varieties with nonisomorphic fundamental groups.

\end{abstract}

Let $V$ be an algebraic variety over $\mathbb{C}$, and let $\tau$ be an
automorphism of $\mathbb{C}$ (as an abstract field). On applying $\tau$ to the
coefficients of the polynomials defining $V$, we obtain a conjugate algebraic
variety $\tau V$ over $\mathbb{C}$. The cohomology groups $H^{i}(V^{\text{an}%
},\mathbb{Q})$ and $H^{i}(\left(  \tau V\right)  ^{\text{an}},\mathbb{Q})$
have the same dimension, and hence are isomorphic, because, when tensored with
$\mathbb{Q}_{\ell}$, they become isomorphic to the \'{e}tale cohomology groups
which are unchanged by conjugation. Similarly, the profinite completions of
the fundamental groups of $V^{\text{an}}$ and $\left(  \tau V\right)
^{\text{an}}$ are isomorphic because they are isomorphic to the \'{e}tale
fundamental groups. However, Serre \cite{se} gave an example in which the
fundamental groups themselves are \emph{not} isomorphic (see \cite{ev} for a
discussion of this and other examples). It seems to have been known (or, at
least, expected) for some time that the theory of Shimura varieties provides
many more examples, but, as far we know, the details have not been written
down anywhere. The purpose of this note is to provide these details.

In the first section, we explain how the semisimple algebraic group attached
to a connected Shimura variety changes when one conjugates the variety, and in
the second section, we apply Margulis's super-rigidity theorem to deduce that
the isomorphism class of the fundamental group changes for large classes of varieties.

We note that all the varieties we consider have canonical models defined over
number fields. Thus, they provide examples of algebraic varieties defined over
number fields with the property that different embeddings of the number field
into $\mathbb{C}$ give complex varieties with nonisomorphic fundamental
groups.\medskip

We assume that the reader is familiar with the basic theory of Shimura
varieties, for example, with the first nine sections of \cite{isv}, and with
the standard cohomology theories for algebraic varieties, as explained, for
example, in the first section of \cite{de}. The identity component of a Lie
group $H$ is denoted $H^{+}$. We use $\approx$ to denote an isomorphism, and
$\simeq$ to denote a canonical isomorphism. We let $U_{1}=\{z\in\mathbb{C}%
\mid|z|=1\}$ regarded as either a real algebraic group or a real Lie group,
and we let $\mathbb{P}$ denote the set of prime numbers $\{2,3,5,7,\ldots\}$.

\section{Conjugation of connected Shimura varieties}

\begin{plain}
\label{p0}Throughout this section, $F$ is a totally real field of finite
degree over $\mathbb{Q}$ which should \emph{not} be thought of as a subfield
of $\mathbb{C}$. For an algebraic group $H$ over $F$, we let $H_{\ast}$ denote
the algebraic group over $\mathbb{Q}$ obtained by restriction of scalars.
Thus, for any $\mathbb{Q}$-algebra $R$, $H_{\ast}(R)=H(F\otimes_{\mathbb{Q}%
}R)$ and $H_{\ast}(\mathbb{R})=\prod_{v\colon F\rightarrow\mathbb{R}}%
H_{v}(\mathbb{R})$ where $H_{v}$ is the algebraic group over $\mathbb{R}$
obtained by extension of scalars $v\colon F\rightarrow\mathbb{R}$. Then $H(F)$
is a subgroup of $H_{\ast}(\mathbb{R})$, and we let $H(F)^{+}=H(F)\cap
H_{\ast}(\mathbb{R})^{+}$.
\end{plain}

\begin{plain}
\label{p0a}By a hermitian symmetric domain, we mean any complex manifold
isomorphic to a bounded symmetric domain. Let $X$ be a hermitian symmetric
domain, and let $\mathrm{Hol}(X)^{+}$ be the identity component of the group
of holomorphic automorphisms of $X$. Let $H$ be a semisimple algebraic group
over $F$, and let $H_{\ast}(\mathbb{R})^{+}\rightarrow\mathrm{Hol}(X)^{+}$ be
a surjective homomorphism with compact kernel. For any torsion-free arithmetic
subgroup $\Gamma$ of $H^{\mathrm{ad}}(F)^{+}$, the quotient manifold
$V=\Gamma\backslash X$ has a unique structure of an algebraic variety
(Baily-Borel, Borel). When the inverse image of $\Gamma$ in $H(F)$ is a
congruence subgroup, we say that the algebraic variety $V$ is of type $(H,X)$.
Note that if $V$ is of type $(H,X)$, then it is also of type $(H_{1},X)$ for
any isogeny $H_{1}\rightarrow H$, in particular, for $H_{1}$ the simply
connected covering group of $H$.
\end{plain}

\begin{theorem}
\label{t1}Let $V$ be a smooth algebraic variety over $\mathbb{C}$, and let
$\tau$ be an automorphism of $\mathbb{C}$. If $V$ is of type $(H,X)$ with $H$
simply connected, then $\tau V$ is of type $(H^{\prime},X^{\prime})$ for some
semisimple algebraic group $H^{\prime}$ over $F$ such that%
\begin{equation}
\left\{
\begin{aligned} H_{v}^{\prime} & \approx H_{\tau\circ v}\text{ for all real primes }v\colon F\rightarrow\mathbb{R}\text{ of }F\text{, and}\\ H_{F_{v}}^{\prime} & \simeq H_{F_{v}}\text{ for all finite primes }v\text{ of }F. \end{aligned}\right.
\label{eq5}%
\end{equation}

\end{theorem}

\begin{example}
\label{x1}Let $B$ be a quaternion algebra over $F$, and let $\inv_{v}%
(B)\in\{0,\frac{1}{2}\}$ be the invariant of $B$ at a prime $v$ of $F$. Recall
that $B$ is determined up to isomorphism as an $F$-algebra by its invariants.
We assume that there exists a real prime of $F$ at which $B$ is indefinite,
and we let $X$ be a product of copies of the complex upper half-plane indexed
by such primes. Let $H=\SL_{1}(B)$. Then $H_{\ast}(\mathbb{R})$ acts in a
natural way on $X$, and for any torsion-free congruence subgroup $\Gamma$ of
$H(F)$, the quotient $V\overset{\text{{\tiny def}}}{=}\Gamma\backslash X$ is a
smooth algebraic variety over $\mathbb{C}$. The theorem shows that, for any
automorphism $\tau$ of $\mathbb{C}$, $\tau V$ arises in a similar fashion from
a quaternion algebra $B^{\prime}$ such that%
\begin{align*}
\inv_{v}(B^{\prime})  &  =\inv_{\tau\circ v}(B)\quad\text{for all infinite
primes }v\text{ of }F\text{, and}\\
\inv_{v}(B^{\prime})  &  =\inv_{v}(B)\quad\text{for all finite primes }v\text{
of }F.
\end{align*}

\end{example}

\begin{remark}
Let $M=\Gamma\backslash X$ be the quotient of a hermitian symmetric domain by
a torsion-free discrete subgroup $\Gamma$ of $\mathrm{Hol}(X)^{+}$. Then $X$ is the
universal covering space of $M$ and $\Gamma$ is the group of covering
transformations of $X$ over $M$, and so, for any $o\in X$ and its image
$\bar{o}$ in $M$, there is an isomorphism $\Gamma\simeq\pi_{1}(M,\bar{o})$
which depends only on the choice of $o$ mapping to $\bar{o}$.
\end{remark}

\begin{remark}
\label{r1}For $V$ as in the theorem, it is possible to describe the
fundamental group of $\left(  \tau V\right)  ^{\text{an}}$ in terms of that of
$V^{\text{an}}$. By assumption, $V=\Gamma\backslash X$ with $\Gamma$ a
torsion-free arithmetic subgroup of $H^{\mathrm{ad}}(F)^{+}\subset\mathrm{Hol}(X)^{+}$
and the inverse image $\tilde{\Gamma}$ of $\Gamma$ in $H(F)$ is a congruence
subgroup, i.e., $\tilde{\Gamma}=H(F)\cap K$ for some compact open subgroup $K$
of $H(\mathbb{A}_{F}^{\infty})$. Similarly, $\tau V=\Gamma^{\prime
}\backslash X^{\prime}$ and the inverse image $\tilde{\Gamma}^{\prime}$ of
$\Gamma^{\prime}$ in $H^{\prime}(F)$ equals $K^{\prime}\cap H^{\prime}(F)$ for
a compact open subgroup $K^{\prime}$ of $H^{\prime}(\mathbb{A}_{F}^{\infty
})$.

It is known that the isomorphisms $H_{F_{v}}\simeq H_{F_{v}}^{\prime}$ in
Theorem \ref{t1} induce an isomorphism $H(\mathbb{A}_{F}^{\infty})\simeq
H^{\prime}(\mathbb{A}_{F}^{\infty})$, and that $K$ maps to $K^{\prime}$
under this isomorphism. If the map $\tilde{\Gamma}\rightarrow\Gamma$ is
surjective, then $\Gamma^{\prime}$ equals the image of $\tilde{\Gamma}%
^{\prime}\overset{\text{{\tiny def}}}{=}H^{\prime}(F)\cap{}K^{\prime}$
 in $H^{\prime\mathrm{ad}}(F)$. We can
always arrange this by choosing $\Gamma$ to be the image of a torsion-free
congruence subgroup of $H(F)$.

In general, the map\footnote{Recall that, for a simply connected semisimple
group $H$, $H(\mathbb{R})$ is connected.} $H(F)\rightarrow H^{\mathrm{ad}%
}(F)^{+}$ is not surjective, and so $\tilde{\Gamma}\rightarrow\Gamma$ need not
be surjective. However, the isomorphisms $H_{F_{v}}\simeq H_{F_{v}}^{\prime}$
define an isomorphism%
\[
\widehat{H^{\mathrm{ad}}(F)^{+}}\rightarrow\widehat{H^{\prime\mathrm{ad}%
}(F)^{+}}%
\]
where the hat means that we are taking the completion for the topologies
defined by the arithmetic subgroups whose inverse images in $H(F)$ or
$H^{\prime}(F)$ are congruence subgroups \cite[8.2]{ms2}. It is known that
$\widehat{\Gamma}$ maps to $\widehat{\Gamma^{\prime}}$ under this isomorphism.
As $\Gamma^{\prime}=\widehat{\Gamma^{\prime}}\cap H^{\prime\mathrm{ad}}%
(F)^{+}$, this determines the fundamental group of $(\tau V)^{\text{an}}$.
\end{remark}

The proof of Theorem \ref{t1} (and Remark \ref{r1}) will occupy the rest of
this section.

\subsection{Preliminary remarks}

\begin{plain}
\label{p1}Let $H$ be a simply connected semisimple algebraic group over $F$,
and let $H_{\ast}(\mathbb{R})^{+}\rightarrow\mathrm{Hol}(X)^{+}$ be a
surjective homomorphism with compact kernel as in (\ref{p0a}). Choose a point
$o\in X$. There exists a unique homomorphism $u\colon U_{1}\rightarrow
H_{\ast}^{\mathrm{ad}}(\mathbb{R})$ such that

\begin{itemize}
\item $u(z)$ fixes $o$ and acts as multiplication by $z$ on the tangent space
at $o$;

\item $u(z)$ projects to $1$ in each compact factor of $H_{\ast}^{\mathrm{ad}%
}(\mathbb{R})$
\end{itemize}

\noindent\noindent\cite[1.9]{isv}. It is possible to recover $X$ and the
homomorphism $H_{\ast}(\mathbb{R})^{+}\rightarrow\mathrm{Hol}(X)^{+}$ from
$(H,u)$ (ibid. 1.21). Moreover, the isomorphism class of the pair
$(H_{\mathbb{R}},u)$ is determined by the isomorphism class of the pair
$(H_{\mathbb{C}},u_{\mathbb{C}})$ (ibid. 1.24). The cocharacter $u_{\mathbb{C}%
}$ of $H_{\mathbb{C}}$ satisfies the following condition (ibid.
1.21):\quote(*) in the action of $\mathbb{G}_{m}$ on $\Lie(G_{\mathbb{C}})$
defined by $\ad\circ u_{\mathbb{C}}$, only the characters $z,1,z^{-1}$ occur.\endquote
\end{plain}

\begin{plain}
\label{p2}Let $k$ be an algebraically closed field of characteristic zero.
Consider a simple adjoint group $H$ over $k$ and a cocharacter $\mu$ of $H$
satisfying (*). Let $T$ be a maximal torus in $H$, let $R$ be the set of
roots of $H$ relative to $T$ (so its elements are characters of $T$), and 
choose a base for $R$. Among the cocharacters of $H$ conjugate to $\mu$, 
there is exactly one $\mu^{\prime}$ that factors through $T$ and is such 
that $\langle\alpha,\mu^{\prime}\rangle\geq0$ for all positive $\alpha\in R$, 
and among the simple roots, there is exactly one $\alpha$ such that 
$\langle\alpha,\mu^{\prime}\rangle\neq0$. The isomorphism class of the pair 
$(H,\mu)$ is determined by the Dynkin diagram of $H$ and the root $\alpha$ 
\cite[1.24 et seq.]{isv}.
\end{plain}

\begin{plain}
\label{p3}Let $(H,\mu)$ be a pair as in (\ref{p2}), and let $\tau\colon
k\rightarrow k^{\prime}$ be a homomorphism of fields. Then $\Lie(\tau
H)\simeq\Lie(H)\otimes_{k,\tau}k^{\prime}$, and this isomorphism defines an
isomorphism from the Dynkin diagram of $\Lie(H)$ to that of $\Lie(\tau H)$
which sends the special root attached to $\mu$ to that attached to $\tau\mu$.
Therefore, if $k=k^{\prime}$, then $(H,\mu)\approx(\tau H,\tau\mu)$.
\end{plain}

\subsection{Conjugation of abelian varieties}

Let $A$ be an abelian variety over an algebraically closed field $k$ of
characteristic zero. The \'{e}tale cohomology groups $H_{\mathrm{et}}%
^{r}(A,\mathbb{Q}_{\ell}(s))$ for $\ell\in\mathbb{P}$ are finite-dimensional
$\mathbb{Q}_{\ell}$-vector spaces and the de Rham cohomology groups
$H_{\mathrm{dR}}^{r}(A)(s)$ are finite-dimensional $k$-vector spaces. For
convenience, we denote these groups by $H_{\ell}^{r}(A)(s)$ and $H_{\infty
}^{r}(A)(s)$ respectively. When $k=\mathbb{C}$, we denote the Betti cohomology
groups $H^{r}(A^{\text{an}},\mathbb{Q}(s))$ by $H_{B}^{r}(A)(s)$.

Recall that there are canonical isomorphisms for $l\in\mathbb{P}\cup
\{\infty\}\cup\{B\}$,%
\begin{equation}
H_{l}^{1}(A^{n})\simeq nH_{l}^{1}(A),\quad H_{l}^{r}(A^{n})\simeq
\bigwedge\nolimits^{r}H_{l}^{1}(A^{n}),\quad H_{l}^{r}(A^{n})(s)\simeq
H_{l}^{r}(A^{n})\otimes_{\mathbb{Q}}\mathbb{Q}(s). \label{eq2}%
\end{equation}
Recall also that, when $k=\mathbb{C}$, there are canonical comparison
isomorphisms%
\begin{equation}
\left\{
\begin{aligned} H_B^{r}(A)(s)\otimes_{\mathbb{Q}}\mathbb{Q}_{\ell} & \overset{\simeq}{\longrightarrow} H_{\ell}^{r}(A)(s),\quad\ell\in\mathbb{P},\\ H_B^{r}(A)(s)\otimes_{\mathbb{Q}}\mathbb{C} & \overset{\simeq}{\longrightarrow}H_{\infty}^{r}(A)(s). \end{aligned}\right.
\label{eq1}%
\end{equation}

For a homomorphism $\tau\colon k\rightarrow k^{\prime}$ of algebraically
closed fields, $\tau A$ denotes the abelian variety over $k^{\prime}$ obtained
by extension of scalars. The morphism $\tau A\rightarrow A$ induces
isomorphisms of vector spaces%
\begin{equation}
\left\{
\begin{aligned} x & \mapsto {}^{\tau}x\colon H_{\ell}^{r}(A)(s)\rightarrow H_{\ell}^{r}(\tau A)(s),\qquad\qquad\ell\in\mathbb{P},\\ x\otimes 1 & \mapsto {}^{\tau}x\colon H_{\infty}^{r}(A)(s)\otimes_{k,\tau}k^{\prime}\rightarrow H_{\infty}^{r}(\tau A)(s). \end{aligned}\right.
\label{eq4}%
\end{equation}

A family
\[
t=(t_{l})_{l\in\mathbb{P}{}\cup\{\infty\}}\in\prod\nolimits_{l\in\mathbb{P}%
}H_{l}^{r}(A)(s)\times H_{\infty}^{r}(A)(s),\quad(r,s)\in\mathbb{N}%
\times\mathbb{Z},
\]
is said to be a Hodge class on $A$ if for some homomorphism $\sigma\colon
k\rightarrow\mathbb{C}$ there exists a class
\[
t_{0}\in H_{B}^{r}(\sigma A)(s)\cap H^{r}(\sigma A,\mathbb{\mathbb{C}%
}(s))^{0,0}%
\]
mapping to $^{\sigma}t_{l}$ under every comparison isomorphism (\ref{eq1}).
The main theorem of \cite{de} says that a Hodge class $t$ on $A$ then has this
property for \textit{every} homomorphism $\sigma\colon k\rightarrow\mathbb{C}$.

Now fix a homomorphism $i\colon F\rightarrow\End^{0}(A)$. Then $H_{l}^{1}(A)$
is a free $F\otimes_{\mathbb{Q}}\mathbb{Q}_{l}$-module (when $l=\infty$,
$\mathbb{Q}_{l}=k$, and when $l=B$, $\mathbb{Q}_{l}=\mathbb{Q}$). We let%
\[
H_{l}^{r}(A)^{\prime}=\bigwedge\nolimits_{F\otimes_{\mathbb{Q}}\mathbb{Q}%
_{l}}^{r}H_{l}^{1}(A),\quad l\in\mathbb{P}\cup\{\infty\}\cup\{B\}\text{.}%
\]
Thus $H_{l}^{r}(A)^{\prime}$ is the largest subspace of $H_{l}^{r}(A)$ with a
natural action of $F$. By a Hodge class on $(A,i)$, we mean a Hodge class
$(t_{l})_{l\in\mathbb{P}\cup\{\infty\}}$ on $A$ such that if $t_{l}$ lies in
the space $H_{l}^{r}(A)(s)$, then it actually lies in the subspace $H_{l}%
^{r}(A^{n})^{\prime}(s)$. We endow $A^{n}$ with the diagonal action of $F$.

Let $t=(t_{l})$ be a Hodge class on $(A,i)$. According to Deligne's theorem
\cite{de}, for any automorphism $\tau$ of $k$, the family $^{\tau}%
t\overset{\textup{{\tiny def}}}{=}({}^{\tau}t_{l})$ is a Hodge class on $\tau
A$. Here $^{\tau}t_{l}$ is the image of $t_{l}$ under the isomorphism
(\ref{eq4}).

We now take $k=\mathbb{C}$.

\subsubsection{Study at a finite prime}

Let $\ell\in\mathbb{P}$. From the isomorphism $F\otimes_{\mathbb{Q}%
}\mathbb{Q}_{\ell}\simeq\prod_{v|\ell}F_{v}$ and the diagonal action of $F$ on
$A^{n}$, we get a decomposition $H_{\ell}^{r}(A^{n})^{\prime}\simeq
\bigoplus\nolimits_{v|\ell}H_{\ell}^{r}(A^{n})_{v}$ in which $H_{\ell}%
^{r}(A^{n})_{v}$ is the subspace on which $F$ acts through $F\rightarrow
F_{v}$. The first comparison isomorphism (\ref{eq1}) for $A^{n}$ induces an
$F\otimes_{\mathbb{Q}}\mathbb{Q}_{\ell}$-linear isomorphism $H_{B}%
^{r}(A^{n})^{\prime}(s)\otimes_{\mathbb{Q}}\mathbb{Q}_{\ell}\simeq H_{\ell
}^{r}(A^{n})^{\prime}(s)$, and hence an $F_{v}$-linear isomorphism
\begin{equation}
H_{B}^{r}(A^{n})^{\prime}(s)\otimes_{F}F_{v}\overset{\simeq}{\longrightarrow
}H_{\ell}^{r}(A^{n})_{v}(s). \label{eq6}%
\end{equation}

The first isomorphism (\ref{eq4}) for $A^{n}$ induces an $F\otimes
_{\mathbb{Q}}\mathbb{Q}_{\ell}$-linear isomorphism $H_{\ell}^{r}%
(A^{n})^{\prime}(s)\rightarrow H_{\ell}^{r}(\tau A^{n})^{\prime}(s)$, and
hence an $F_{v}$-linear isomorphism%
\begin{equation}
H_{\ell}^{r}(A^{n})_{v}(s)\overset{\simeq}{\longrightarrow}H_{\ell}^{r}(\tau
A^{n})_{v}(s). \label{eq7}%
\end{equation}

\subsubsection{Study at an infinite prime}

We identify the infinite primes of $F$ with homomorphisms $v\colon
F\rightarrow\mathbb{R}\subset\mathbb{C}$. Note that $F\otimes_{\mathbb{Q}%
}\mathbb{R}$ and $F\otimes_{\mathbb{Q}}\mathbb{C}$ are products of
copies of $\mathbb{R}$ and $\mathbb{C}$ respectively indexed by the real
primes of $F$. The second comparison isomorphism (\ref{eq1}) for $A^{n}$
induces an $F\otimes_{\mathbb{Q}}\mathbb{C}$-linear isomorphism%
\[
H_{B}^{r}(A^{n})^{\prime}(s)\otimes_{\mathbb{Q}}\mathbb{C}\overset{\simeq
}{\longrightarrow}H_{\infty}^{r}(A^{n})^{\prime}(s),
\]
and hence a $\mathbb{C}$-linear isomorphism%
\begin{equation}
H_{B}^{r}(A^{n})^{\prime}(s)\otimes_{F,v}\mathbb{C}\overset{\simeq
}{\longrightarrow}H_{\infty}^{r}(A^{n})_{v}^{\prime}(s) \label{eq8}%
\end{equation}
for each real prime $v$ of $F$.

The second isomorphism (\ref{eq4}) for $A^{n}$ induces an $F\otimes
_{\mathbb{Q}}\mathbb{C}$-isomorphism
\[
H_{\infty}^{r}(A^{n})^{\prime}(s)\otimes_{\mathbb{C},\tau}\mathbb{C}%
\rightarrow H_{\infty}^{r}(\tau A^{n})^{\prime}(s),
\]
and hence a $\tau$-linear isomorphism%
\begin{equation}
H_{\infty}^{r}(A^{n})_{v}^{\prime}(s)\rightarrow H_{\infty}^{r}(\tau
A^{n})_{\tau\circ v}^{\prime}(s) \label{e2}%
\end{equation}
for each real prime $v$ of $F$.

We now assume that $A$ admits complex multiplication, so that there exists a
CM algebra $E$ containing $F$ and a homomorphism $E\rightarrow\End^{0}(A)$
making $H_{1}(A,\mathbb{Q})$ into a free $E$-module of rank $1$. Let
$T=(\mathbb{G}_{m})_{E/F}$ (restriction of scalars, so $T$ is a torus over
$F$), and let $T_{v}=T\otimes_{F,v}\mathbb{R}$. The action of $E$ on
$H_{1}(A,\mathbb{Q})$ defines an action of $T$ on $H_{1}(A,\mathbb{Q})$
and on its dual $H_{B}^{1}(A)$. Recall that%
\begin{equation}
H_{B}^{1}(A)\otimes_{\mathbb{Q}}\mathbb{C}\simeq H^{1,0}(A)\oplus H^{0,1}(A),\quad
H^{1,0}\overset{\text{{\tiny def}}}{=}H^{0}(A,\Omega_{A}^{1}),\quad
H^{0,1}\overset{\text{{\tiny def}}}{=}H^{1}(A,\mathcal{O}_{A}), \label{eq10}%
\end{equation}
from which we deduce a decomposition%
\[
H_{B}^{1}(A)\otimes_{F,v}\mathbb{C}\simeq H_{v}^{1,0}\oplus H_{v}^{0,1}.
\]
There is a cocharacter $\mu_{A}$ of $T_{\ast}$ such that, for $z\in
\mathbb{C}^{\times}$, $\mu_{A}(z)$ acts on $H^{1,0}$ as $z^{-1}$ and on
$H^{0,1}$ as $1$. This decomposes into cocharacters $\mu_{v}$ of
$T_{v\mathbb{C}}$ such, for $z\in\ \mathbb{C}^{\times}$, $\mu_{v}(z)$ acts on
$H_{v}^{1,0}$ as $z^{-1}$ and on $H_{v}^{0,1}$ as $1$.

The decomposition (\ref{eq10}) is algebraic:%
\[
H_{\infty}^{1}(A)\simeq H^{0}(A,\Omega_{A}^{1})\oplus H^{1}(A,\mathcal{O}%
_{A})
\]
(Zariski cohomology of algebraic sheaves). The action of $E$ on $A$ defines an
action of $E$ on $\tau A$, and hence an action of $T$ on its cohomology
groups. There is a commutative diagram%
\[
\begin{CD}
H_{\infty}^{1}(A) @>{\simeq}>> H^{0}(A,\Omega_{A}^{1})\oplus H^{1}(A,\mathcal{O}_{A})\\
@V{\simeq}V{\tau\text{-linear}}V@V{\simeq}V{\tau\text{-linear}}V\\
H_{\infty}^{1}(\tau A) @>>> H^{0}(\tau A,\Omega_{A}^{1})\oplus H^{1}(\tau
A,\mathcal{O}_{A})
\end{CD}
\]
which is compatible with the actions of $T$, from which it follows that
\begin{equation}
\mu_{\tau A}=\tau\mu_{A} \label{eq11}%
\end{equation}
($\tau\mu_{A}$ makes sense because both $\mathbb{G}_{m}$ and $T_{\ast}$ are
$\mathbb{Q}$-tori). Let $\mu_{v}^{\prime}$ be the cocharacter of $T_{v}$
arising from the action of $T$ on the cohomology of $\tau A$. Then
(\ref{eq11}) shows that%
\begin{equation}
\mu_{\tau\circ v}^{\prime}=\tau\mu_v. \label{eq12}%
\end{equation}

\subsection{Proof of Theorem \ref{t1}}

Let $V$ be an algebraic variety over $\mathbb{C}$ of type $(H,X)$ in the sense
of (\ref{p0a}) with $H$ simply connected. Theorem 1.1 of \cite{mi} shows
that, for any automorphism $\tau$ of $\mathbb{C}$, the conjugate $\tau V$ of
$V$ is of type $(H^{\prime},X^{\prime})$ where the pair $(H^{\prime}%
,X^{\prime})$ has a precise description, which we now explain.

Let $\mathbb{Q}^{\mathrm{al}}$ denote the algebraic closure of $\mathbb{Q}$ in
$\mathbb{C}$, and let $\mathsf{M}$ be the category of motives for absolute
Hodge classes generated by the Artin motives over $\mathbb{Q}$ and the abelian
varieties over $\mathbb{Q}$ that become of CM-type over $\mathbb{Q}%
^{\mathrm{al}}$ \cite[\S 6]{dm}. The functor sending a variety $W$ over
$\mathbb{Q}$ to the Betti cohomology of $W_{\mathbb{C}}$ defines a fibre
functor $\omega_{B}$ on $\mathsf{M}$. There is an exact sequence%
\begin{equation}
1\longrightarrow S^{\circ}\longrightarrow S\overset{\pi}{\longrightarrow
}\Gal(\mathbb{Q}^{\mathrm{al}}/\mathbb{Q})\longrightarrow1 \label{eq3}%
\end{equation}
in which $S$ is the pro-algebraic group attached to the Tannakian category
$\mathsf{M}$ and its fibre functor $\omega_{B}$, $\Gal(\mathbb{Q}%
^{\mathrm{al}}/\mathbb{Q})$ is the constant pro-algebraic group attached to
$\omega_{B}$ and the Tannakian subcategory of Artin motives over $\mathbb{Q}$,
and $S^{\circ}$ is the protorus attached to $\omega_{B}$ and the Tannakian
category of CM motives over $\mathbb{Q}^{\mathrm{al}}$ (equivalently, over
$\mathbb{C}$) (ibid. 6.28). The homomorphisms in the sequence are defined by
the obvious tensor functors. The groups $S^{\circ}$ and $S$ are respectively
the connected Serre group and the Serre group.

Now let $\mathsf{M}_{F}$ be the category of motives in $\mathsf{M}$ with
coefficients in $F$. Thus, an object of $\mathsf{M}_{F}$ is an object $X$ of
$\mathsf{M}$ together with a homomorphism $F\rightarrow\End(X)$ (see
\cite[2.1]{de2} for an alternative interpretation). Betti cohomology defines a
fibre functor $\omega_{B}$ on $\mathsf{M}_{F}$ to $F$-vector spaces, and the
exact sequence%
\[
1\longrightarrow S_{F}^{\circ}\longrightarrow S_{F}\overset{\pi}%
{\longrightarrow}\Gal(\mathbb{Q}^{\mathrm{al}}/\mathbb{Q})_{F}\longrightarrow
1
\]
attached to $\mathsf{M}_{F}$ and $\omega_{B}$ is obtained from (\ref{eq3}) by
extension of scalars $\mathbb{Q}\rightarrow F$ \cite[3.12]{dm}.

With $(H,X)$ as in the first paragraph, there exists a $o\in X$ that is fixed
by $T(\mathbb{R})^{+}$ for some maximal torus $T\subset H^{\mathrm{ad}}$. Let
$u\colon U_{1}\rightarrow T_{\ast\mathbb{R}}\subset H_{\ast\mathbb{R}%
}^{\mathrm{ad}}$ be the homomorphism corresponding to $(H,X)$ as in
(\ref{p1}), and let $\mu$ be the cocharacter $u_{\mathbb{C}}$ of
$T_{\ast\mathbb{C}}$. There is a unique homomorphism $\rho\colon S_{F}^{\circ
}\rightarrow T$ sending the universal cocharacter of $S_{F}^{\circ}$ to $\mu$
\cite[\S 1]{ms1}. For $\tau\in\Gal(\mathbb{Q}^{\mathrm{al}}/\mathbb{Q})$, the
inverse image $P(\tau)$ of $\tau$ in $S_{F}$ is an $S_{F}^{\circ}$-torsor, and
the pair $(H^{\prime},\mu^{\prime})$ is obtained from $(H,\mu)$ by twisting by
$P(\tau)$:%
\begin{equation}
\left(  H^{\prime},\mu^{\prime}\right)  =(H,\mu)\wedge^{S_{F}^{\circ}}P(\tau).
\label{eq9}%
\end{equation}
When $F=\mathbb{Q}$, this is proved in \cite[Theorem 1.1]{mi}.\footnote{In
fact, the results of \cite{mi} are stated in terms of Langlands's Taniyama
group rather than the Serre group. However, effectively they are proved for
the Serre group, and then Deligne's theorem \cite{de4} is used to replace the
Serre group with the Taniyama group.} The statement with $F$ is, in fact,
weaker than the statement with $\mathbb{Q}$, and follows from it. Note that a
variety $V$ of type $(H,X)$ with $H$ a semisimple algebraic group over $F$ is
also of type $(H_{\ast},X)$, where $H_{\ast}$ is now a semisimple algebraic
group over $\mathbb{Q}$, and that the $\rho$ for $(H,X)$ is related to the
$\rho$ for $(H_{\ast},X)$ by the canonical isomorphism%
\[
\Hom(S_{F}^{\circ},T)\simeq\Hom(S^{\circ},T_{\ast}).
\]
Using this, it is easy to deduce the statement for $F$ from that for
$\mathbb{Q}$.

We now prove Theorem \ref{t1} for a variety $V$ of type $(H,X)$. The
homomorphism $\rho\colon S_{F}^{\circ}\rightarrow T\subset H$ factors through
an algebraic quotient of $S_{F}^{\circ}$. This allows us to replace
$\mathsf{M}_{F}$ in the above discussion with the subcategory generated by the
Artin motives and a single (very large) abelian variety $A$ equipped with an
$F$-action $i\colon F\rightarrow\End^{0}(A)$. Then $S_{F}^{\circ}$ is replaced
with the subgroup \textrm{MT}$(A,i)$ of $\GL_{F}(H_{B}^{1}(A))\times
\GL_{1}(F(1))$ of elements fixing all Hodge classes on $(A_{\mathbb{Q}%
^{\mathrm{al}}},i)$ and its powers. Moreover, $P(\tau)$ is the \textrm{MT}%
$(A,i)$-torsor such that, for any $F$-algebra $R$, $P(\tau)(R)$ consists of
the isomorphisms%
\[
a\colon H_{B}^{1}(A_{\mathbb{C}})\otimes_{F}R\rightarrow H_{B}^{1}(\tau
A_{\mathbb{C}})\otimes_{F}R
\]
sending $t$ to $^{\tau}t$ for all Hodge classes $t$ on $A_{\mathbb{Q}%
^{\mathrm{al}}}$ \cite[6.29]{dm}.

Let $v$ be a finite prime of $F$ dividing $\ell$. Define $a_{v}$ to be the
isomorphism making the following diagram commute:%
\[
\begin{CD}
H_{B}^{1}(A)\otimes_{F}F_{v} @>{a_{v}}>>H_{B}^{1}(\tau A)\otimes_{F}F_v\\
@V{\simeq}V{\text{(\ref{eq6})}}V @V{\simeq}V{\text{(\ref{eq6})}}V\\
H_{\ell}^{1}(A)_{v} @>{\simeq}>{\text{(\ref{eq7})}}>H_{\ell}^{1}(\tau A)_{v}%
\end{CD}
\]
Then there are commutative diagrams
\[
\begin{CD}
H_{B}^{r}(A^n)^{\prime}(s)\otimes_{F}F_{v} @>{\text{\textquotedblleft}a_{v}\text{\textquotedblright}}>>H_{B}^{r}(\tau A^n)^{\prime}(s)\otimes_{F}F_v\\
@V{\simeq}V{\text{(\ref{eq6})}}V @V{\simeq}V{\text{(\ref{eq6})}}V\\
H_{\ell}^{r}(A^n)_{v}(s) @>{\simeq}>{\text{(\ref{eq7})}}>H_{\ell}^{r}(\tau A^n)_{v}(s)%
\end{CD}
\]
for all $(r,n,s)\in\mathbb{N}\times\mathbb{N}\times\mathbb{Z}$, where
\textquotedblleft$a_{v}$\textquotedblright\ is the map deduced from $a_{v}$ by
means of the isomorphisms (\ref{eq2}). This shows that $a_{v}\in P(\tau
)(F_{v})$ and so it defines an isomorphism $H_{F_{v}}\simeq H_{F_{v}}^{\prime
}$ (see (\ref{eq9})).

Before considering the real primes, we make another observation. Let
$\mathsf{T}$ be a Tannakian category, let $\omega$ be a fibre functor on
$\mathsf{T}$ with values in the category of vector spaces over some field $k$.
The group of tensor automorphisms $\underline{\Aut}^{\otimes}(\omega)$ of
$\omega$ is an affine group scheme $G$ over $k$. A homomorphism $\tau\colon
k\rightarrow k^{\prime}$ defines a fibre functor
\[
^{\tau}\omega\colon X\rightsquigarrow\omega(X)\otimes_{k,\tau}k^{\prime}%
\]
on $\mathsf{T}$, and the group $G^{\prime}$ of tensor automorphisms of
$^{\tau}\omega$ is canonically isomorphic to $\tau G$ \cite[3.12]{dm}. Now
suppose $k=k^{\prime}$. Then the functor $P=\underline{\Hom}^{\otimes}%
(\omega,{}^{\tau}\omega)$ of tensor homomorphisms from $\omega$ to $^{\tau
}\omega$ is a torsor for $G$, and $^{\tau}G$ is the twist of $G$ by this
torsor:%
\[
^{\tau}G=G\wedge^{G}P.
\]

We now consider a real prime $v\colon F\rightarrow\mathbb{R}\subset
\mathbb{C}$. In this case, we define $a_{v}$ to be the $\tau$-linear
isomorphism making the following diagram commute:%
\[
\begin{CD}
H_{B}^{1}(A)\otimes_{F,v}\mathbb{C}
@>{a_{v}}>>H_{B}^{1}(\tau A)\otimes_{F,\tau\circ v}\mathbb{C}\\
@V{\simeq}V{\text{(\ref{eq8})}}V @V{\simeq}V{\text{(\ref{eq8})}}V\\
H_{\infty}^{1}(A)_{v}
@>{\simeq}>{\text{(\ref{e2})}}>H_{\infty}^{1}(\tau A)_{\tau\circ v}.%
\end{CD}
\]
Again, it gives commutative diagrams
\[
\begin{CD}
H_{B}^{r}(A^n)^{\prime}(s)\otimes_{F,v}\mathbb{C}
@>{\text{\textquotedblleft}a_{v}\text{\textquotedblright}}>>H_{B}^{r}(\tau A^n)^{\prime}(s)\otimes_{F,\tau\circ v}\mathbb{C}\\
@V{\simeq}V{\text{(\ref{eq8})}}V @V{\simeq}V{\text{(\ref{eq8})}}V\\
H_{\infty}^{r}(A^n)_{v}(s)@>{\simeq}>{\text{(\ref{e2})}}>H_{\infty}^{r}(\tau A^n)_{\tau\circ v}(s)%
\end{CD}
\]
Thus $a_{v}$ gives an isomorphism $\left(  \tau H_{v}\right)  _{\mathbb{C}%
}\overset{\simeq}{\longrightarrow}\left(  H_{\tau\circ v}^{\prime}\right)
_{\mathbb{C}}$. Because of (\ref{eq12}), this isomorphism sends $\tau\mu
_{v}$ to $\mu_{\tau\circ v}^{\prime}$. In other words,
\[
(\left(  H_{\tau\circ v}^{\prime}\right)  _{\mathbb{C}},\mu_{\tau\circ v
}^{\prime})\simeq(\left(  \tau H_{v}\right)  _{\mathbb{C}},\tau\mu
)\overset{\text{(\ref{p3})}}{\simeq}(\left(  H_{v}\right)  _{\mathbb{C}}%
,\mu).
\]
We now apply (\ref{p2}) to deduce that $H_{\tau\circ v}^{\prime}\approx H_{v}$.

This completes the proof of Theorem \ref{t1}, but it remains to explain Remark
\ref{r1}. The statement in (\ref{r1}) concerning $K$ and $K^{\prime}$ follows
immediately from Theorem 1.1 of \cite{mi}. For the stronger statement
concerning $\widehat{\Gamma}$ and $\widehat{\Gamma^{\prime}}$, we need to
appeal to Proposition 6.1 of \cite{mi}.

\begin{remark}
\label{r2}The proof we have given of Theorem \ref{t1} makes use of the main
theorem of \cite{mi}, and hence of the main theorem of \cite{ka}. For Shimura
varieties of abelian type, it is possible to give a more direct proof of the
theorem that avoids using these results.
\end{remark}

\section{The examples}




In this section, we find examples of connected Shimura varieties $S$ over
$\mathbb{C}$ and field automorphisms $\tau$ of $\mathbb{C}$ such that
$S^{\mathrm{an}}$ and $(\tau S)^{\mathrm{an}}$ have nonisomorphic fundamental
groups. Such $S$ will be attached to a pair $(H,X)$ as in the first section,
where $H$ is an absolutely simple algebraic group over a totally real number
field $F\neq\mathbb{Q}$. As $H_{\ast}$ will be $\mathbb{Q}$-simple with 
$\dim_{\mathbb{C}} X>1$, (i) when $S$ is compact, it will be of general type 
because the canonical bundle will be ample, while (ii) when $S$ is noncompact, 
it will be an open subvariety (with complement of codimension $\geq2$) of a
normal projective variety (the Baily-Borel compactification) with ample
canonical bundle.

We prepare necessary tools in the first two subsections, spell out the
algorithm for construction in the third, and then deal with some examples.

\subsection{Super-rigidity}

The crucial ingredient in our construction of nonhomeomorphic conjugate
Shimura varieties is the super-rigidity of Margulis:

\begin{theorem}[Margulis]
\label{suprig} Let $F$ be a totally real number field and let $H$ and
$H^{\prime}$ be two absolutely simple, simply connected algebraic groups over
$F$ with the (real) rank of $H_{\ast}\otimes_{\mathbb{Q}}\mathbb{R}$ at least
$2$. Suppose that $\Gamma$ (resp. $\Gamma^{\prime}$) is a lattice of $H_{\ast
}(\mathbb{Q})=H(F)$ (resp. $H_{\ast}^{\prime}(\mathbb{Q})$). If $\Gamma$ and
$\Gamma^{\prime}$ are isomorphic, then there exists a field automorphism
$\sigma$ of $F$ such that $H^{\prime}$ is isomorphic to $\sigma H$.
\end{theorem}

\begin{proof}
This follows from \cite[Theorem C, p. 259]{ma}, together with the fact that an
epimorphism between $H^{\prime}$ and a conjugate of $H$ is necessarily an
isomorphism, because both are assumed to be simply connected.
\end{proof}

\begin{corollary}
\label{suprigcor} Let $F$, $H$ and $H^{\prime}$ be as in Theorem
\ref{suprig} and suppose there is an isomorphism
\begin{equation}
\rho:H\otimes_{F}\mathbb{A}_{F}^{\infty}\cong H^{\prime}\otimes_{F}%
\mathbb{A}_{F}^{\infty}. \label{adelicisom}%
\end{equation}
Let $K$ be an open compact subgroup of $H\left(  \mathbb{A}_{F}^{\infty
}\right)  $ such that the intersections
\[
\Gamma:=K\cap H(F)\text{ and }\Gamma^{\prime}:=\rho(K)\cap H^{\prime}(F)
\]
are both torsion-free. If $H^{\prime}$ is not isomorphic to $\sigma H$ for any
$\sigma\in\Aut(F)$ as an $F$-group, then $\Gamma$ and $\Gamma^{\prime}$ are
not isomorphic. $\square$
\end{corollary}

\subsection{Automorphism-free number fields and marking finite places}

In applying Corollary \ref{suprigcor}, it will be convenient if the only field
automorphism of $F$ is the identity. We recall how to make such number fields.

\begin{proposition}
For any integer $d\ge3$, there exist totally real number fields $F$ of degree
$d$ over $\mathbb{Q}$ not admitting any field automorphism other than the identity.
\end{proposition}

\begin{proof}
Let $K$ be a Galois extension of $\mathbb{Q}$ with Galois group $S_{d}$, and
let $F$ be the fixed field of the subgroup $S_{d-1}$ of $S_{d}$ of elements
fixing $1$. The automorphism group of $F$ is the normalizer of $S_{d-1}$ in
$S_{d}$ modulo $S_{d-1}$, which is trivial. Thus, to prove the proposition, it
suffices to find a polynomial $g$ of degree $d$ in $\mathbb{Q}[X]$ whose
Galois group is the full group $S_{d}$ of permutations of its roots and which
splits over $\mathbb{R}$, because then the field $F$ generated by any root of
$g$ will have the required properties. Such a polynomial can be constructed by
a standard argument, which we recall.

Let $p_{1},p_{2},p_{3}$ be distinct primes, and let $f_{1},f_{2},f_{3}%
\in\mathbb{Z}[X]$ be monic polynomials of degree $d$ such that $f_{1}$ is
irreducible modulo $p_{1},$ $f_{2}$ is the product of a linear polynomial and
an irreducible polynomial modulo $p_{2}$, and $f_{3}$ is the product of an
irreducible quadratic polynomial and distinct irreducible polynomials of odd 
degree modulo $p_{3}$. 
Finally, let $f_{\infty}\in\mathbb{R}[X]$ be monic of degree $d$ with $d$
distinct real roots. It follows from the weak approximation theorem applied to
the coefficients of the polynomials that there exists a monic polynomial
$g\in\mathbb{Q}[X]$ such that $g-f_{i}\in p_{i}\mathbb{Z}_{p_{i}}[X]$ for
$i=1,2,3$ and $g$ is sufficiently close to $f_{\infty}$ in the real topology
that it also has $d$ distinct real roots. The Galois group of $g$ is a
transitive subgroup of $S_{d}$ containing a transposition and a $(d-1)$ cycle,
and so equals $S_{d}$.
\end{proof}

When $F$ admits nontrivial automorphisms, the following proposition will be useful.

\begin{proposition}
\label{nontrivaut} Let $F$ be a number field and let $p$ be a rational prime
number which splits completely in $F$. If a field automorphism $\sigma$ of $F$
fixes one $p$-adic place of $F$, then $\sigma= 1$.
\end{proposition}

\begin{proof}
Choose a Galois closure $K/\mathbb{Q}$ of $F/\mathbb{Q}$ with group $G$.
Denote by $X$ the set of $p$-adic places of $K$; as $p$ also splits completely
in $K$, $X$ can be identified with the set
\[
\Hom\left(  K,\mathbb{Q}_{p}\right)
\]
of ring homomorphisms of $K$ into $\mathbb{Q}_{p}$, and $G$ acts simply
transitively on $X$ by composition.

Let $H$ be the Galois group of $K/F$. Then the set of $p$-adic places of $F$
is identified with the set $X/H$ of $H$-orbits in $X$, and $\Aut(F)$ with
$N_{G}(H)/H$, where $N_{G}(H)$ denotes the normalizer of $H$ in $G$. In this
identification, the action of $\Aut(F)$ on the $p$-adic places of $F$
corresponds to the induced action of $N_{G}(H)/H$ on $X/H$. Thus the problem
reduces to the following trivial statement: If $H$ is a subgroup of a group
$G$ and if $X$ is a principal homogeneous space for $G$, then an element of
$N_{G}(H)$ leaving stable one $H$-orbit in $X$ necessarily belongs to $H$.
\end{proof}

\subsection{Construction}

Let $F$ be a totally real field. We explain how to construct 

\begin{itemize}
\item pairs $(H,X)$ as in (\ref{p0a}) with $H$ satisfying the hypotheses of
Theorem \ref{suprig},  and 

\item automorphisms $\tau$ of $\mathbb{C}$ such that any group $H^{\prime}$
satisfying (\ref{eq1}) of Theorem \ref{t1} will not be isomorphic to $\sigma
H$ for any automorphism $\sigma$ of $F$. 
\end{itemize}

\noindent Corollary \ref{suprigcor} then implies that, for any sufficiently 
small congruence subgroup $\Gamma$ of $H(F)$, $S\overset{\textup{{\tiny def}}}%
{=}\Gamma\backslash X$ is an algebraic variety over $\mathbb{C}$ whose
fundamental group is not isomorphic to that of $\tau S$.

Let $T$ be one of $A_{n}$ ($n\geq1$), $B_{n}$ ($n\geq2$), $C_{n}$ ($n\geq3$),
$D_{n}$ ($n\geq4$), $E_{6}$ or $E_{7}$.

For each real prime $v$ of $F$, choose a simply connected simple algebraic
group $H(v)$ over $\mathbb{R}$ of type $T$ that is either compact or is 
associated with a hermitian symmetric domain $X(v)$ (i.e. 
$H(v)^{\mathrm{ad}}(\mathbb{R})^{+}=\mathrm{Hol}(X(v))^{+}$). 
Recall that, up to isomorphism, there is exactly one compact possibility for 
$H(v)$ and at least one hermitian possibility (and more than one unless 
$T=A_{1}$, $A_{2}$, $B_{n}$, $C_{n}$, $D_{4}$, $E_{6}$ or $E_{7}$); 
see \cite[p. 518]{he} or \cite[1.3.9]{de1}. We choose the $H(v)$ so that

\begin{enumerate}
\item[(i)] for some $v\neq v^{\prime}$, $H(v)\not \approx H(v^{\prime})$, and

\item[(ii)] the sum of the ranks $\sum_{v|\infty}\mathrm{rk}_{\mathbb{R}}H(v)$
is at least $2$.
\end{enumerate}

\noindent Let $H(\infty)$ (resp. $X$) be the product of $H(v)$'s (resp.
$X(v)$'s) as $v$ runs over all the real primes of $F$ (resp. over the real
primes of $F$ at which $H(v)$ is noncompact). There is a surjective
homomorphism $H(\infty)(\mathbb{R})^{+}\rightarrow\mathrm{Hol}(X)^{+}$ with
compact kernel.

When $F$ admits nontrivial automorphisms, we also need to mark finite places
in some way (cf. Remark \ref{galoisevasion} below). In this case, choose a
rational finite prime $p$ that splits completely in $F$ (such primes abound by
the Chebotarev density theorem) and a prime $v_{0}$ of $F$ lying above $p$.
For each prime $v$ of $F$ lying over $p$, choose an absolutely simple, simply
connected $\mathbb{Q}_{p}$-group $H(v)$ such that $H(v)\not \approx H(v_{0})$
if $v\neq v_{0}$.

When $F$ admits no nontrivial automorphism, we take $S$ to be the set of real
primes of $F$, and otherwise we take it to be the set of real or $p$-adic
primes of $F$. According to a theorem of Borel and Harder \cite[Theorem B]%
{bh}, there exists an absolutely simple, simply connected group $H$ over $F$
such that for each $v\in S$, the localization $H_{v}$ of $H$ at $v$ is
isomorphic to $H(v)$. Then $H$ satisfies the hypotheses of Theorem
\ref{suprig} and there is a surjective homomorphism $H(\mathbb{R}%
)^{+}\rightarrow\mathrm{Hol}(X)^{+}$ with compact kernel.

Let $\tau$ be an automorphism of $\mathbb{C}$ such that, for at least one real
place $v_{1}$, $H(v_{1})$ and $H(\tau\circ v_{1})$ are not
isomorphic.\footnote{The choice of the $H(v)$ defines a partition of set of
real primes of $F$, in which two primes lie in the same set if the groups are
isomorphic. The hypothesis on $\tau$ is that it not preserve this partition.}
Given $(H,X)$, such a $\tau$ always exists because of the condition (i) and
the fact that $\Aut(\mathbb{C})$ acts transitively on the set of real primes
of $F$. Conversely, given a $\tau\in\Aut(\mathbb{C})$ acting non-trivially on
the real places of $F$ and a type $T$, we can always choose a pair $(H,X)$ of
type $T$ satisfying this hypothesis.

It remains to prove, with this choice for $(H,X)$, a group $H^{\prime}$
satisfying the isomorphisms (\ref{eq5}) of Theorem \ref{t1} can not be
isomorphic to $\sigma H$ for any automorphism $\sigma$ of $F$. If $\Aut(F)=1$,
then $\sigma H$ is not isomorphic to $H^{\prime}$ because, $\sigma$ being $1$,
that would imply that $H(v_{1})\approx H^{\prime}(v_{1})\approx H(\tau\circ
v_{1})$, contrary to hypothesis. When $\Aut(F)\neq1$, the condition at the
$p$-adic primes force $\sigma=1$ in view of Proposition \ref{nontrivaut}, and
the same argument applies.

The dimension of the resulting Shimura varieties is equal to $\dim
_{\mathbb{C}} X = \sum_{v} \dim_{\mathbb{C}} X(v)$, where $v$ runs over the
infinite places of $F$ with $H(v)$ noncompact; the table on \cite[p. 518]{he}
gives the real dimension of $X(v)$.

As noted above, there is only one noncompact real form associated with a
hermitian symmetric domain for the types $A_{1}$, $A_{2}$, $B_{n}$, $C_{n}$,
$D_{4}$, $E_{6}$ and $E_{7}$, and $H(v)$ is compact for at least one real
place $v$. This means that the resulting Shimura varieties are necessarily
projective (\cite[3.3(b)]{isv}). For the types $A_{n}(n\geq3)$ and
$D_{n}(n\geq5)$, however, we can construct noncompact examples. Furthermore,
for $A_{n}$, we can find these examples (compact or otherwise) among PEL
moduli spaces of abelian varieties. These will be explained in the subsections
below, after we deal with the quaternionic examples a bit more concretely.

\subsection{Quaternionic Shimura varieties}

We start with quaternionic Shimura varieties as in Example \ref{x1}; so we
denote by $B$ a quaternion algebra over a totally real number field $F$ of
degree $d\ge3$ over $\mathbb{Q}$ and by $\Sigma(B)$ the set of places $v$ of
$F$ at which $B$ is ramified (i.e., $\inv_{v}(B) = 1/2$). The last set
decomposes into $\Sigma(B) = \Sigma(B)_{f} \amalg\Sigma(B)_{\infty}$ of finite
and infinite places of ramification. We let $H:= \SL_{1} (B)$; the real rank
of $H_{\ast}$ is equal to $d - |\Sigma(B)_{\infty}|$.

\subsubsection{Automorphism-free $F$}

Suppose $F$ admits only the identity as automorphism. Choose $B/F$ and
$\tau\in\Aut(\mathbb{C})$ such that (i) $\tau\Sigma(B)_{\infty} \neq
\Sigma(B)_{\infty}$ and (ii) $d-|\Sigma(B)_{\infty}| \ge2$; we can either
first choose $B/F$ that satisfies (ii) and then choose a suitable $\tau$, or
choose $\tau\in\Aut(\mathbb{C})$ acting non-trivially on the real places of
$F$ and then choose a suitable $B/F$.

Then $H$ and the group $H^{\prime}$ attached to $(H, X)$ and $\tau$ satisfy
the conditions in Corollary \ref{suprigcor}.

\subsubsection{General $F/\mathbb{Q}$}

If $F$ does admit an automorphism other than the identity, we need to ``mark''
some finite places. Choose a rational prime $p$ over which $F$ splits
completely, and then $B/F$ and $\tau\in\Aut(\mathbb{C})$ such that in addition
to the two conditions (i) and (ii) as above, we have (iii) there is exactly
one (out of $d$) $p$-adic place $v_{p}$ in $\Sigma(B)$.

Then, again, $H$ and $H^{\prime}$ for this choice of $B$ and $\tau$ satisfy
the conditions in Corollary \ref{suprigcor}.

\begin{remark}
\label{galoisevasion} To see why one does need some extra argument marking
finite places, let $F$ be Galois over $\mathbb{Q}$, and suppose that for all
the rational prime numbers $p$, any two $p$-adic places $v$ and $w$ are
simultaneously contained in $\Sigma(B)$ (i.e., either both are in or both are
out); for example, suppose that $B$ is unramified at every finite place. Then,
even if $B$ satisfies the two conditions (i) and (ii), for \textit{any} choice
of $\tau\in\Aut(\mathbb{C})$, $B$ and $B^{\prime}$ will be isomorphic as
$\mathbb{Q}$-algebras (though not as $F$-algebras). Consequently, $H_{\ast}$
and $H^{\prime}_{\ast}$ will be isomorphic as $\mathbb{Q}$-groups, and as such
will admit isomorphic lattices.
\end{remark}

The dimension of the resulting quaternionic Shimura varieties is equal to
$d-|\Sigma(B)_{\infty}|$, which can be any integer $\geq2$. As $\Sigma
(B)_{\infty}$ contains at least one, but misses at least two, of infinite
places, it follows that all these varieties (a) are necessarily compact and
(b) cannot have rational weights, and a fortiori cannot be a classifying space
of rational Hodge structures \cite[5.24 and p. 332]{isv}.

\subsection{Unitary PEL Shimura varieties}

Now we turn to certain unitary Shimura varieties, which are moduli spaces of
polarized abelian varieties with endomorphism and level structures. We will
keep the following notation throughout: $F$ is a totally real number field of
degree $d\ge2$, $E/F$ is a totally imaginary quadratic extension, and $n\ge2$
is an integer.

We consider hermitian matrices $A \in M_{n}(E)$, that is, those with $A^{\ast}
= A$. Recall that over the real numbers, there are $\lfloor n/2 \rfloor+1$
distinct isomorphism classes of special unitary groups $SU(p, q)$ with $p\ge
q\ge0$ and $p+q = n$, while over any $p$-adic local field, there are one or
two isomorphism classes of special unitary groups, according as $n$ is odd or
even. Moreover, all of these can be defined by \textit{diagonal} hermitian
matrices. By the approximation theorem applied to the diagonal entries (in
$F$) of $A$, for any finite subset $\Sigma$ of places of $F$ and any
prescribed type of special unitary group $H(v)$ for each $v\in\Sigma$ (there
is only one choice if $v$ is finite and $n$ is odd), there is a diagonal
matrix $A \in M_{n}(F)$ such that the localization at each $v\in\Sigma$ of the
special unitary group $SU(A)$ over $F$ attached to $A$ is isomorphic to $H(v)$.

\subsubsection{Automorphism-free $F$}

Suppose $\Aut(F) = 1$ (hence $d=[F:\mathbb{Q}]\ge3$). For each infinite place
$v$ of $F$, choose a pair $(p_{v}, q_{v})$ with $p_{v}\ge q_{v} \ge0$ and
$p_{v} + q_{v} = n$, subject to the conditions that (i) not all of the pairs
should be the same and (ii) $\sum_{v | \infty} q_{v} \ge2$. As
$\Aut(\mathbb{C})$ acts transitively on the infinite places, we can choose
$\tau\in\Aut(\mathbb{C})$ such that $p_{v_{0}} \neq p_{\tau\circ v_{0}}$ for
some infinite place $v_{0}$. When $n\ge4$, we can even make $q_{v} > 0$ for
all $v|\infty$. Choose a diagonal matrix $A \in M_{n} (F) $ such that the
associated special unitary group $H:=SU(A)$ over $F$ has the chosen type
$SU(p_{v}, q_{v})$ at each infinite place $v$ of $F$.

Then there is a PEL Shimura datum $(G,X)$ such that $G^{\der}\cong H_{\ast}$,
and the groups $H$ and $H^{\prime}$ satisfy the conditions in Corollary
\ref{suprigcor}. The resulting varieties have dimension equal to
\[
\sum_{v|\infty}p_{v}q_{v}.
\]
They are compact if $q_{v}=0$ for at least one $v|\infty$. When $q_{v}>0$ for
all $v|\infty$, one can make noncompact examples by letting $A$ contain each
of $1$ and $-1$ as diagonal entries at least
\[
q_{0}:=\min_{v|\infty}q_{v}>0
\]
times; then the associated special unitary group contains the $q_{0}$-fold
product of $\SL_{2}$ over $F$, hence is isotropic, and $S$ will be noncompact
(cf. \cite[3.3(b)]{isv}). The smallest dimension of noncompact examples
obtained this way is $10=3\cdot1+3\cdot1+2\cdot2$.

\subsubsection{General $F/ \mathbb{Q}$, with $n$ even}

Now drop the condition on $\Aut(F)$, and assume that $p$ is a rational prime
number which (a) splits completely in $F$ and (b) every $p$-adic place $v$ of
$F$ remains inert in $E$ (given $F$, such an $E/F$ can be obtained by taking
the compositum of $F$ with a suitable imaginary quadratic field). In case $n$
is even, we can mark $p$-adic places, since there are two distinct types of
special unitary groups for the extension $E_{v} / \mathbb{Q}_{p}$.

More precisely, choose the diagonal matrix $A\in M_{n}(F)$ which satisfies, in
addition to the conditions (i) and (ii) as above, (iii) there is one $p$-adic
place $v_{0}$ of $F$ such that $SU(A)_{v_{0}}$ is not isomorphic to
$SU(A)_{v}$ for every other $p$-adic place $v\neq v_{0}$. Put $H:= SU(A)$ and
choose $\tau\in\Aut(\mathbb{C})$ that doesn't preserve the signatures at
infinity. Then the resulting groups $H$ and $H^{\prime}$ satisfy the conditions.

The sufficient condition for compactness and the formula for dimension as
above still hold in this case. The smallest dimension of noncompact examples
in this case is $7 = 3\cdot1 + 2\cdot2$ over real quadratic fields.

\begin{remark}
Using division algebras (instead of the matrix algebra $M_{n}(E)$) with
involution of the second kind extending the complex conjugation on $E/F$, one
can make inner forms of unitary groups $H$ and choose $\tau\in\Aut(\mathbb{C}%
)$ so that the resulting groups $H$ and $H^{\prime}$ satisfy the conditions in
Corollary \ref{suprigcor}, regardless of the parity of $n$ (see \cite[\S 2]%
{cl} for details). In this case, $H$ and $H^{\prime}$ will be anisotropic and
the corresponding Shimura varieties will be projective. They still have
interpretation as PEL moduli varieties.
\end{remark}

\subsection{Noncompact examples of type $D$}

Consider the case of type $D_{n}$, with $n\geq4$. There are two series (called
$BD\,\,I(q=2)$ and $D\,\,III$ \cite[p. 518]{he}) of hermitian symmetric
domains, coinciding exactly when $n=4$. Thus, there are two choices of
noncompact forms when $n\geq5$, and one can ask if this extra choice allows
noncompact examples. In this section, we construct such examples; these will
\textit{not} be of abelian type (see \cite[p. 331]{isv}).

Let $F$ be a number field and let $B$ be a quaternion division algebra over
$F$, with the main involution $\ast= \ast_{B}$. Then the matrix algebra $M:=
M_{n}(B)$ also has an involution $\ast_{M}$, mapping an element $\left(
a_{ij}\right)  _{1\le i, j\le n}$ to $\left(  a^{\ast}_{ji}\right)  _{1\le i,
j\le n}$. An invertible skew-hermitian matrix $A$ defines another involution
$\ast_{A}$ on $M$ by conjugation:
\begin{align*}
g^{\ast_{A}} := A \cdot g^{\ast_{M}} \cdot A^{-1},
\end{align*}
and one can form the special unitary group over $F$, given by
\[
SU(A)(R) := \left\{  g\in M_{n}(B) \otimes_{F} R : g^{\ast_{A}} \cdot g = 1
\text{ and } \mathrm{Nrd}^{M}_{F} (g) = 1 \right\}
\]
for an $F$-algebra $R$. Over an algebraic closure $\overline{F}$, $SU(A)$
becomes isomorphic to $SO_{2n}$.

If $v$ is a place of $F$ at which $B$ is unramified, there is an isomorphism
\begin{equation}
\xymatrix{ \varphi_v : B\otimes_F F_v \ar[r]^-{\sim} & M_2(F_v) }
\label{locsplt}%
\end{equation}
of $F_{v}$-algebras, under which $\ast_{B}$ corresponds to the involution
\begin{equation}
\dagger:x\mapsto J\cdot{\vphantom{x}}^{t}{x}\cdot J^{-1}, \label{daggerinv}%
\end{equation}
where $t$ stands for the transpose of a $2\times2$ matrix and $J\in
M_{2}(F_{v})$ is $t$-skew-symmetric (i.e., ${\vphantom{J}}^{t}{J}=-J$). As any
two non-degenerate alternating forms on $F_{v}^{\oplus2}$ are equivalent, by
composing $\varphi_{v}$ with an inner automorphism, we may assume that $J=%
\begin{pmatrix}
0 & 1\\
-1 & 0
\end{pmatrix}
$.

For an integer $n\ge1$, put $J(n):= \mathrm{diag} (J, \cdots, J) \in M_{n}
\left(  M_{2}(F_{v}) \right)  $. Then under the isomorphism (that we still
call $\varphi_{v}$) of $F_{v}$-algebras induced by $\varphi_{v}$:
\[
\xymatrix{
M\otimes_F F_v = M_n (B) \otimes_F F_v \cong M_n (B \otimes_F F_v) \ar[r]^-{\sim} & M_n \left( M_2(F_v) \right),
}
\]
the involution $\ast_{M} \otimes1$ corresponds to the involution on the
target
\[
\dagger: \left(  a_{ij} \right)  _{1\le i, j \le n} \mapsto\left(  J
\cdot{\vphantom{a_{ji}}}^{t}{a_{ji}} \cdot J^{-1} \right)  _{1\le i, j\le n}.
\]
Let $I_{1,1}:=
\begin{pmatrix}
1 & 0\\
0 & -1
\end{pmatrix}
\in M_{2}(F_{v})$. The matrix $J(n)$ (resp. the element $I_{1,1}$) is
$\dagger$-skew-symmetric, and so defines an involution and the special
unitary group over $F_{v}$ as above, such that
\begin{align*}
SU(J(n)) (F_{v})  &  \cong\left\{  g\in M_{2n} (F_{v}) : {\vphantom{g}}^{t}{g}
\cdot g = 1 \text{ and } \det(g) = 1 \right\}  ,\\
SU(I_{1,1}) (F_{v})  &  \cong\left\{  g\in M_{2} (F_{v}) : {\vphantom{g}}%
^{t}{g} \cdot%
\begin{pmatrix}
0 & 1\\
1 & 0
\end{pmatrix}
\cdot g =
\begin{pmatrix}
0 & 1\\
1 & 0
\end{pmatrix}
\text{ and } \det(g) = 1 \right\}  .
\end{align*}

Now we assume that $n\ge5$, that $F$ is a totally real number field of degree
$d\ge3$ with $\Aut(F) = 1$, and that the sets $S(B)_{\mathbb{H}}$ and
$S(B)_{\mathbb{R}}$ of definite and indefinite infinite places of $B/F$,
respectively, are both non-empty, so that there is a $\tau\in\Aut(\mathbb{C})$
not leaving $S(B)_{\mathbb{H}}$ stable. (See below for removing the assumption
$\Aut(F) = 1$.)

For each $v\in S(B)_{\mathbb{R}}$, fix isomorphisms $\varphi_{v}$ as in
(\ref{locsplt}) satisfying the condition (\ref{daggerinv}) with the standard
$J$ as above. By using the approximation theorem, we can choose a $\ast
$-skew-hermitian element $a_{1}\in B$ whose image under $\varphi_{v}$ is very
close to $I_{1,1} \in M_{2}(F_{v}) \cong M_{2}(\mathbb{R})$ for each $v\in
S(B)_{\mathbb{R}}$. Put $a_{2} := -a_{1}$. Then choose $\ast$-skew-hermitian
elements $a_{3}, \cdots, a_{n} \in B$ such that $\varphi_{v}(a_{i})$ is very
close to $J$ for each $v\in S(B)_{\mathbb{R}}$ and every $i=3, \cdots, n$.
Then the matrix $A=\mathrm{diag} (a_{1}, \cdots, a_{n}) \in M_{n}(B)$ is
skew-hermitian and invertible.

Let $H$ be the universal covering of $SU(A)$. The matrix $A$ was crafted so
that $H_{v}$ is of type $BD\,\, I(q=2)$ (resp. of type $D\,\, III$) for $v\in
S(B)_{\mathbb{R}}$ (resp. $v\in S(B)_{\mathbb{H}}$). As $\tau$ does not leave
$S(B)_{\mathbb{H}}$ stable and as we assumed $\Aut(F) = 1$, the group
$H^{\prime}$ attached to $H$ and $\tau$ satisfy the conditions in Corollary
\ref{suprigcor} (the real rank of the form $BD\,\, I(q=2)$ (resp. $D\,\, III$)
is $2$ (resp. $\lfloor n/2 \rfloor\ge2$), and the rank condition (ii) in the
construction is automatically met).

The group $H$ is isotropic over $F$---hence the resulting Shimura varieties
are noncompact---because $SU(A)$ contains $SU(a_{1}, -a_{1})$, which in turn
contains $SO(1, -1)$ over $F$ (since $a_{1}$ commutes with $F$), which is
isomorphic to the split torus over $F$. The Shimura varieties attached to $H$
have dimension
\[
|S(B)_{\mathbb{R}}| \cdot(2n-2) + |S(B)_{\mathbb{H}}| \cdot n(n-1)/2.
\]

Again, when $F$ (of degree $\geq2$ over $\mathbb{Q}$) has non-trivial
automorphisms, we can mark $p$-adic places, where $p$ is a rational prime such
that $F/\mathbb{Q}$ splits completely above $p$ and $B/F$ is unramified at
every $p$-adic place of $F$: Choose $a_{1}$ and $a_{2}$ just as above, but
this time include the $p$-adic places into consideration when applying the
approximation theorem to $a_{3},\cdots,a_{n}$, so as to distinguish the
special orthogonal (and hence the spin) groups at $p$-adic places; then use
Proposition \ref{nontrivaut}.

\subsection{Varieties defined over number fields}

As noted in the introduction, all of the varieties we consider have canonical
models defined over number fields. Briefly, given $(H,X)$, we need to choose a
(nonconnected) Shimura datum $(G,Y)$ such that $G$ is a reductive group with
derived group $H_{\ast}$ and $Y$ is disjoint union of hermitian symmetric domains,
one component of which is $X$. Then the Shimura variety defined by $(G,Y)$ has
a canonical model over its reflex field $E(G,Y)$, which is a number field, and
each of the varieties $\Gamma\backslash X$ has a canonical model over an
abelian extension of $E(G,Y)$ which can be described by class field theory.
See \cite{isv}. We explain this in the case of quaternionic Shimura varieties.

Thus, let $B$ be a quaternion algebra over a totally real field $F$. Then $B$
defines a partition of the real primes of $F$%
\begin{equation}
\Hom(F,\mathbb{R})=\Sigma_{\infty}\sqcup\Sigma_{\infty}^{\prime} \label{eq13}%
\end{equation}
with $\Sigma_{\infty}$ the set of primes where $B$ ramifies. Let $G$ be the
reductive group over $\mathbb{Q}$ with $G(\mathbb{Q})=B^{\times}$, and let $Y$
be a product of copies of $\mathbb{C}\smallsetminus\mathbb{R}$ indexed by
$\Sigma_{\infty}^{\prime}$ (assumed to be nonempty). Then $(G,Y)$ is a Shimura datum. The Galois group of
$\mathbb{Q}^{\mathrm{al}}$ acts on $\Hom(F,\mathbb{Q}^{\mathrm{al}}%
)\simeq\Hom(F,\mathbb{R})$, and its reflex field $E=$ $E(G,Y)$ is the fixed
field of the set of automorphisms of $\mathbb{Q}^{\mathrm{al}}$ stabilizing
the sets $\Sigma_{\infty}$ and $\Sigma_{\infty}^{\prime}$. See \cite[12.4(d)]{isv}

The norm map defines a homomorphism $\mathrm{Nm}\colon G\rightarrow\left(
\mathbb{G}_{m}\right)  _{F/\mathbb{Q}}$. The theory of Shimura varieties
shows that the Shimura variety attached to $(G,Y)$ has a canonical model over
$E(G,Y)$. Moreover, its set of connected components is $\left(\mathbb{A}_{F}%
^{\infty}\right)^{\times}/\overline{F^{\times}}$, and the theory provides a 
specific (reciprocity) homomorphism $r\colon\left(\mathbb{A}_{E}^{\infty}%
\right)^{\times}/E^{\times}\rightarrow\left(\mathbb{A}_{F}^{\infty}\right)%
^{\times}/\overline{F^{\times}}$. Let $K$ be a compact open subgroup of 
$G(\mathbb{A}^{\infty})$ sufficiently small that the congruence subgroup 
$\Gamma=K\cap H_{\ast}(\mathbb{Q})$ of $H_{\ast}(\mathbb{Q})$ is torsion free. 
Use the norm to map $K$ to a subgroup of $\left(\mathbb{A}_{F}^{\infty}\right)%
^{\times}/\overline{F^{\times}}$, and let $K^{\prime}$ be the inverse image 
of this subgroup in $\left(\mathbb{A}_{E}^{\infty}\right)^{\times}/E^{\times}$. 
The complex algebraic variety $\Gamma\backslash X$ has a canonical model $V$ 
over the abelian extension $E^{\prime}$ of $E$ corresponding by class field 
theory to $K^{\prime}$.

Now assume that $\left\vert \Sigma_{\infty}^{\prime}\right\vert \geq2$, and let
$\sigma_{1}$ and $\sigma_{2}$ be homomorphisms $E^{\prime}\rightarrow
\mathbb{C}$. If $F$ has no nontrivial automorphisms, then $\sigma_{1}V$ and
$\sigma_{2}V$ have nonisomorphic fundamental group whenever $\sigma_{1}%
|E\neq\sigma_{2}|E$ (because then there exists a real prime $v$ of $F$ such
that $\sigma_{1}\circ v$ and $\sigma_{2}\circ v$ lie in different subsets of
the partition (\ref{eq13})). If $F$ has nontrivial automorphisms, then, as
before, we need to impose a condition on a prime $p$ in order to be able to 
draw the same conclusion.

\begin{remark}
Finally we remark that we have proved that two arithmetic groups $\Gamma$ and
$\Gamma^{\prime}$ are not isomorphic by showing that they live in different
algebraic groups so that we can apply Margulis super-rigidity. However, of
course, lattices $\Gamma$ and $\Gamma^{\prime}$ may be nonisomorphic even when
they lie in the same algebraic group. Since the theory of Shimura varieties
gives an explicit description of $\Gamma^{\prime}$ in terms of $\Gamma$ and
$\tau$ (see \ref{r1}), this suggests that it may be possible to enlarge our
class of examples even further.
\end{remark}

\medskip

\noindent Department of Mathematics, University of Michigan, Ann Arbor, MI 48109-1043, USA.\\
E-mail: \url{jmilne@umich.edu}

\medskip\noindent Department of Mathematics, MIT, Cambridge, MA 02139-4307, USA.\\
E-mail: \url{junecuesuh@math.mit.edu}

\end{document}